\documentclass[a4paper,11pt]{article}

\usepackage[top=3.0cm,bottom=3.0cm,left=2.5cm,right=2.5cm]{geometry}

\usepackage{setspace}
\usepackage{lmodern}

\usepackage{graphicx}
\usepackage{amsfonts}
\usepackage{amssymb}
\usepackage{amsmath}
\usepackage{amsthm}
\usepackage[english]{babel}

\usepackage{hyperref}
\usepackage{cases}
\usepackage{authblk}
\usepackage{mathtools}
\usepackage{color}
\usepackage{ifpdf}
\usepackage{subfigure}
\usepackage[graphicx]{realboxes}
\usepackage{array}
\usepackage{diagbox}
\usepackage{rotating}
\usepackage{multirow}
\usepackage{booktabs}
\usepackage{tikz}
\usetikzlibrary{matrix,positioning}
\usetikzlibrary{calc}

\newtheorem{lemma}{Lemma}[section]
\newtheorem{theorem}[lemma]{Theorem}

\newtheorem{claim}{\noindent Claim}[section]

\begin{document}

\setstretch{1.35} 
	
\title{Optimal chromatic bound for ($P_2\cup P_4$, HVN)-free graphs}

\begingroup
\renewcommand\thefootnote{}
\footnotetext{*\;Corresponding author. }
\footnotetext{Email: lchendh@connect.ust.hk (L. Chen), hwanghj@connect.ust.hk (H. Wang).}
\endgroup
{\author[1]{Lizhong Chen\textsuperscript{1} \quad Hongyang Wang\textsuperscript{*}}

\affil{Department of Mathematics,

Hong Kong University of Science and Technology,

Clear Water Bay, Hong Kong}

\date{\today}
\maketitle

\begin{abstract}

The HVN is a graph formed by removing two edges incident to the same vertex from the complete graph $K_5$. In this paper, we prove that every ($P_2\cup P_4$, HVN)-free graph $G$ satisfies $\chi(G)\leq\lceil\frac{4}{3}\omega(G)\rceil$ when $\omega(G)\ge4$, where $\chi(G)$ and $\omega(G)$ denote the chromatic number and clique number of $G$, respectively. Furthermore, this bound is optimal for every $\omega(G)\ge4$. Constructions demonstrating the optimality of the bound are provided. Our work unifies several previously known results on $\chi$-binding functions for several graph classes~\cite{P2P4DIA,P2P3, OURS,RANCHEN,2K2HVN,P6D2,P2P3DIA}.

\end{abstract}

				

\section{Introduction}

All graphs considered in this paper are finite and simple. We follow \cite{WEST} for undefined notations and terminology. Let $G$ be a graph with vertex set $V(G)$ and edge set $E(G)$. Two vertices $u,v\in V(G)$ are \emph{adjacent} if and only if $uv \in E(G)$. We denote by $u\sim v$ if $u,v$ are adjacent and $u\nsim v$ if otherwise. A vertex $v$ and an edge $e$ are said to be \emph{incident} with each other if $v$ is an endpoint of $e$. For a graph $G$, the \emph{complement} $\overline{G}$ is the graph with vertex set $V(G)$ and $uv \in \overline{G}$ if and only if $uv \notin G$. For two graphs $H$ and $G$, $H$ is an \emph{induced subgraph} of $G$ if $V(H)\subseteq V(G)$, and $uv \in E(H)$ if and only if $uv \in E(G)$. We say that a graph $G$ \emph{contains} $H$ if $G$ has an induced subgraph isomorphic to $H$, and say that $G$ is $H$-\emph{free} if otherwise. For a family of graphs $\{H_1,H_2,\ldots\}$, $G$ is $(H_1,H_2,\ldots)$-free if G is $H$-free for every $H\in\{H_1,H_2,\ldots\}$.  

Let $P_n$ and $K_n$ denote the \emph{path}, and \emph{complete graph} on $n$ vertices respectively. The \emph{HVN} (an abbreviation of the German \emph{Haus vom Nikolaus}) is a graph formed by removing two edges incident to the same vertex from the complete graph $K_5$. For two vertex-disjoint graph $G_1$ and $G_2$, the \emph{union} $G_1\cup G_2$ is the graph with $V(G_1\cup G_2)=V(G_1)\cup V(G_2)$ and $E(G_1\cup G_2)=E(G_1)\cup E(G_2)$. We denote by $G[X]$ the subgraph of $G$ induced by a subset $X \subseteq V(G)$, and when the context is clear, we identify $X$ with $G[X]$. A set $X$ is called a \emph{clique} (resp.\ \emph{stable set}) if $G[X]$ is a complete graph (resp.\ has no edges). The \emph{clique number} of $G$, denoted $\omega(G)$, is the size of a largest clique in $G$; we write simply $\omega$ when the graph is clear from context.
\input{P2P4_HVN.tpx}

A \emph{$k$-colouring} is a function $c: V\rightarrow\{1,2,\ldots,k\}$ such that $c(u)\neq c(v)$ whenever $u$ and $v$ are adjacent in $G$. For a subset $A \subseteq V(G)$, we use $c(A)$ to denote the set of colours assigned to the vertices in $A$ under the colouring $c$. A graph is \emph{k-colourable} if it admits a k-colouring. The smallest integer $k$ such that a given graph $G$ is $k$-colourable is called its \emph{chromatic number}, denoted by $\chi(G)$. Obviously, $\chi(G)\geq\omega(G)$. A graph $G$ is called \emph{perfect} if we have $\chi(H)=\omega(H)$ for every induced subgraph $H$ of $G$. 

Gy\'{a}rf\'{a}s \cite{GY} introduced the notion of $\chi$-boundedness as a natural extension of perfect graphs in 1975. A family $\mathcal{G}$ of graphs is said to be $\chi$-\emph{bounded} if there exists a real-valued function $f$ such that $\chi(G)\leq f(\omega(G))$ holds for every graph $G\in \mathcal{G}$. The function $f$ is called a \emph{binding function} for $\mathcal{G}$. Clearly, the class of perfect graphs is $\chi$-bounded and the binding function is the identity function $f(x)=x$. 

We briefly review some results on $\chi$-boundedness. A \emph{hole} is an induced cycle of length at least four. An \emph{antihole} is the complement of a hole. A hole or antihole is \emph{odd} or \emph{even} if it is of odd or even length, respectively. In their seminal paper \cite{SPGT}, Chudnovsky, Robertson, Seymour and Thomas proved the famous Strong Perfect Graph Theorem: a graph is perfect if and only if it is (odd hole, odd antihole)-free. For the class of graphs that forbid only odd holes, a $\chi$-binding function is known to exist; however, the best known such function is double-exponential~\cite{ODDHOLE}. On the other hand, if we only forbid even holes, a linear binding function $f(x)=2x-1$ exists~\cite{EVENHOLE}, that is, every even hole-free graph $G$ satisfies $\chi(G)\leq 2\omega(G)-1$. The existence of $\chi$-binding functions for classes of graphs without holes of various lengths has attracted considerable interest in recent years.

Another line of research is the study of $H$-free graphs for a fixed graph $H$. A classical result of Erd\H{o}s~\cite{ERDOS} shows that for any two positive integers $k,l\geq3$, there exists a graph with $\chi(G)\geq k$ and no cycles of length less than $l$. i.e. the class of $H$-free graphs is not $\chi$-bounded if $H$ contains a cycle. The famous Gy\'{a}rf\'{a}s--Sumner conjecture~\cite{GY, SUMNER} states that the converse is also true. This motivates us to study the chromatic number of $F$-free graphs, where $F$ is a forest (a disjoint union of trees). Gy\'{a}rf\'{a}s~\cite{GY2} proved the conjecture for $F=P_t$.

\begin{theorem}\cite{GY2}
    Every $P_t$-free graph $G$ satisfies $\chi(G)\leq (t-1)^{\omega(G)-1}$.
\end{theorem}

Note that this binding function is exponential, it is natural to ask whether the exponential bound can be improved to a polynomial bound for $P_t$-free graphs. This has proven to be a challenging problem, and little progress has been made thus far. Since $P_t$ is a connected graph and Gy\'{a}rf\'{a}s--Sumner conjecture concerns forests, this naturally leads us to ask what happens when we forbid a disconnected forest, such as $P_2 \cup P_t$.

Indeed, the problem of colouring $(P_2\cup P_t)$-free graphs for $t \in \{2,3,4\}$, along with their subclasses, has been studied extensively in recent years. The case $t=2$ has been well studied; see~\cite{SURVEY} for a comprehensive survey. For $t \in \{3,4\}$, Bharathi and Choudum~\cite{P2P3} established the best-known binding function for $(P_2\cup P_3)$-free graphs, which also applies to the $(P_2\cup P_4)$-free class.

\begin{theorem}\cite{P2P3}\label{LP2P4}
For any $(P_2\cup P_3)$-free or $(P_2\cup P_4)$-free graph $G$, $\chi(G)\leq \frac{\omega(G)(\omega(G)+1)(\omega(G)+2)}{6}$.
\end{theorem}

Following this work, subsequent research has pursued improved bounds for specific subclasses. We investigate the class of $(P_2 \cup P_4$, HVN)-free graphs, covering several important subclasses, particularly ($P_2 \cup P_3$, diamond)-free graphs, ($P_2 \cup P_4$, diamond)-free graphs, and ($P_2 \cup P_3$, HVN)-free graphs. Accordingly, our literature review will focus on these three subclasses, while also noting key results for HVN-free graphs.

For the class of ($P_2\cup P_3$, diamond)-free graphs, Bharathi and Choudum~\cite{P2P3} showed that $\chi(G)\leq4$ when $\omega(G)=2$, a bound which is tight, and that the graph is perfect for $\omega(G)\geq5$. Later, Karthick and Mishra~\cite{P6D2} proved the optimal bound $\chi(G)\leq6$ for $\omega(G)=3$, and Prashant, Francis and Raj~\cite{P2P3DIA} established $\chi(G)=4$ for $\omega(G)=4$. Several other subclasses of $(P_2\cup P_3, R)$-free graphs have also been investigated, for forbidden graphs $R$ such as the crown~\cite{CROWN}, gem~\cite{gem}, house~\cite{house}, and HVN~\cite{P2P3DIA}.

For subclasses of $P_2 \cup P_4$-free graphs, Chen and Zhang~\cite{RANCHEN} established chromatic bounds for ($P_2\cup P_4$, diamond)-free graphs: $\chi(G)\leq4$ for $\omega(G)=2$, $\chi(G)\leq7$ for $\omega(G)=3$, $\chi(G)\leq9$ for $\omega(G)=4$, and $\chi(G)\leq2\omega(G)+1$ for $\omega(G)\geq5$. In the same work, they also derived that every ($P_2\cup P_4$, gem)-free graph $G$ satisfies $\chi(G)\leq3\omega(G)-2$, and that every ($P_2\cup P_4$, butterfly)-free graph $G$ satisfies $\chi(G)\leq\frac{1}{2}(\omega(G)^2+3\omega(G)-2)$. Subsequent work by Angeliya, Huang and Karthick~\cite{P2P4DIA} improved these bounds for ($P_2\cup P_4$, diamond)-free graphs with $\omega(G)\geq3$, demonstrating that $\chi(G)\leq\max\{6,\omega(G)\}$; their work also settled the case $\omega(G)=4$ by proving $\chi(G)=4$. Most recently, Chen and Wang~\cite{OURS} proved that every ($P_2\cup P_4$, diamond)-free graph $G$ satisfies $\chi(G)=\omega(G)$ for $\omega(G)\ge4$, providing a complete solution for the $\chi$-binding function of this class.

For HVN-free graphs, several $\chi$-binding functions have been established. Karthick and Mishra~\cite{2K2HVN} proved that every $(2K_2$, HVN)-free graph satisfies $\chi(G) \le \omega(G) + 3$. In the same work~\cite{P2P3DIA}, Prashant, Francis, and Raj extended their analysis to $(P_2 \cup P_3$, HVN)-free graphs, obtaining the optimal bound $\chi(G) \leq \omega(G) + 1$ for $\omega(G) \geq 4$. Further results include those of Song and Xu~\cite{ODDHVN}, who demonstrated that every (odd hole, HVN)-free graph $G$ satisfies $\chi(G) \le \omega(G) + 1$, and Xu~\cite{P5HVN}, who showed that $(P_5$, HVN)-free graphs satisfy $\chi(G) \le \max\{\min\{16, \omega(G)+3\}, \omega(G)+1\}$.
\begin{figure}
    \centering
        \includegraphics[width=0.35\linewidth]{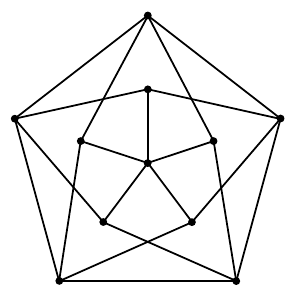}
        \caption{The Mycielski--Gr\"otzsch graph.}
    \label{GRO}
\end{figure}
\subsection{Our Contributions}

In this paper, we extend the chromatic bounds known for ($2K_2$, diamond)-free graphs~\cite{P2P3}, ($2K_2$, HVN)-free graphs~\cite{2K2HVN}, ($P_2\cup P_3$, diamond)-free graphs~\cite{P2P3,P6D2,P2P3DIA}, ($P_2\cup P_3$, HVN)-free graphs~\cite{P2P3DIA}, and ($P_2\cup P_4$, diamond)-free graphs~\cite{P2P4DIA,OURS,RANCHEN} by proving the following $\chi$-binding functions for ($P_2\cup P_4$, HVN)-free graphs.

\begin{theorem}\label{THM2}

Let $G$ be a ($P_2\cup P_4$, HVN)-free graph. Then
\begin{displaymath}
\chi(G)\leq\left\{ \begin{array}{ll}
4 & \textrm{if $\omega(G)=2$}\\
10 & \textrm{if $\omega(G)=3$}\\
\lceil\frac{4}{3}\omega(G)\rceil  & \textrm{if $\omega(G)\geq4$}\\
\end{array} \right.
\end{displaymath}
Moreover, the bounds are tight for every $\omega(G)\neq 3$.
\end{theorem}

For $\omega(G)=2$, the bound $\chi(G)=4$ is attained by the Mycielski--Gr\"otzsch graph (see Figure~\ref{GRO}). While for $\omega(G)\geq 4$, the optimality of the bound $\chi(G)\leq\lceil\frac{4}{3}\omega(G)\rceil$ is established at the end of Section 4 through a constructive proof: we exhibit a class of graphs with clique number $\omega$ on $2\omega^2$ vertices that attains this bound.

\section{Preliminaries}

For a given positive integer $n$, we use $[n]$ to denote the set $\{1,2,\ldots,n\}$. For a vertex $v\in V(G)$, the \emph{neighbourhood} of $v$, denoted by $N_G(v)$, is the set of all vertices adjacent to $v$, and can be simplified to $N(v)$ when the graph $G$ is clear from context.  For a subset $S\subseteq V(G)$, the neighbourhood of $S$, denoted by $N_G(S)$ (or $N(S)$), is the set of all vertices adjacent to some vertex in $S$. The graph obtained by deleting a vertex set $S \subseteq V(G)$ and all incident edges is denoted by $G - S$. For two disjoint vertex sets $A$ and $B$, we let $[A, B]$ denote the set of edges between $A$ and $B$. We say that $A$ is \emph{complete} (resp. \emph{anticomplete}) to $Y$ if $|[X,Y]|=|X||Y|$ (resp. $[X,Y]=\emptyset$). A graph is called \emph{$k$-partite} if its vertex set can be partitioned into $k$ \emph{parts} such that no two vertices within the same part are adjacent. A $2$-partite graph is called \emph{bipartite}. A $k$-partite graph is \emph{complete} if every two vertices from different parts are adjacent. A \emph{matching} in $G$ is a set of edges without common vertices. The vertices incident to the edges of a matching $M$ are \emph{saturated} by $M$, and a \emph{perfect matching} is a matching that saturates every vertex of $G$.

The following lemmas will be used several times in the sequel.

\begin{lemma}\cite{P4}\label{P4}
    Every $P_4$-free graph is perfect.
\end{lemma}

The next lemma is known as Hall's theorem in bipartite graphs.

\begin{lemma}\cite{HALL}\label{HALL}
Let $G$ be a bipartite graph with bipartition $(X,Y)$. Then $G$ has a matching that saturates $X$ if and only if $|N_G(S)|\geq|S|$ for all $S\subseteq X$.
\end{lemma}

\begin{lemma}\label{L0}
Let $G$ be a graph containing two disjoint cliques  $X$ and $Y$ such that $\min\{|X|,|Y|\}\ge2$ and $\max\{|X|,|Y|\}\ge3$. If $[X,Y]$ induces a nonempty matching, 
then $G[X\cup Y]$ contains a $P_4$.
\end{lemma}

\begin{proof}
By symmetry, assume that $|Y|\ge3$. 
Let $\{x_1, x_2\}\subseteq X$ and $\{y_1, y_2, y_3\}\subseteq Y$. 
Since $[X,Y]\neq\emptyset$ and $[X,Y]$ induces a nonempty matching,
we may assume that $x_1\sim y_1$ and $x_2\nsim y_3$. 
Then $\{x_2, x_1, y_1, y_3\}$ induces a $P_4$.
\end{proof}

\section{The Structure of ($P_2 \cup P_4$, HVN)-Free Graphs}

Let G be a ($P_2 \cup P_4$, HVN)-free graph with $\omega \geq 2$. By Lemma~\ref{LP2P4}, if $\omega=2$, then $\chi(G)\leq 4$; if $\omega =3$, then $\chi (G)\leq 10$.

Hence, for the remainder of this section, we may assume that $\omega \geq 4$. Let $A$ be an induced complete $\omega$-partite subgraph of $G$ of maximum order with parts $A_1, A_2, \ldots, A_\omega$, and let $H=G-A$.

\begin{lemma}\label{LH1}
If $\omega\geq 4$, then for every vertex $v\in V(H)$, $N_A(v)$ intersects at most one part of $A$.
\end{lemma}
\begin{proof}
Suppose, to the contrary, that there exists $v \in V(H)$ such that $x, y \in N_A(v)$ with $x \in A_1$ and $y \in A_2$. Since $G$ is HVN-free  and $\omega \geq 4$, the vertex $v$ must be complete to $\omega - 3$ parts and anticomplete to the remaining part among $A_3, A_4, \ldots, A_\omega$. Hence, by the same argument, $v$ is complete to both $A_1$ and $A_2$. 
Consequently, $\{v\} \cup V(A)$ induces a complete $\omega$-partite subgraph of $G$, contradicting the maximality of $A$.
\end{proof}

By Lemma~\ref{LH1}, we may partition $V(H)$ into the following subsets, where $i \in [\omega]$.
\begin{align*}
    C_i &=\{v\in V(H) \mid N_A(v)\subseteq A_i\text{ \,and\, } N_A(v)\neq \emptyset  \},\\
    C_0 &=\{v\in V(H) \mid N_A(v)=\emptyset\}.
\end{align*}

\begin{lemma}\label{LH2}
For all $1\leq i\neq j\leq \omega$, every subgraph induced by a subset of $V(H)\setminus (C_i\cup C_j)$ is $P_4$-free and hence perfect.
\end{lemma}
\begin{proof}
This follows from the fact that $G$ is $(P_2\cup P_4)$-free and that
$V(H)\setminus (C_i\cup C_j)$ is anticomplete to $\{a_i,a_j\}$, where $a_i\in A_i$ and $a_j\in A_j$.
\end{proof}

Suppose that $\omega (H)=k$. The next lemma provides an upper bound on the chromatic number of $G$ in terms of $\omega$ and $k$.

\begin{lemma}\label{LH3}
If $G$ is a ($P_2\cup P_4$, HVN)-free graph with $\omega \geq 4$, then $\chi(G)\leq \max\{2k, \omega\}$.
\end{lemma}
\begin{proof}
Let $X=C_0\cup C_1\cup \cdots\cup C_{\lfloor\frac{\omega}{2}\rfloor}$ and $Y=C_{\lfloor\frac{\omega}{2}\rfloor+1}\cup C_{\lfloor\frac{\omega}{2}\rfloor+2}\cup \cdots \cup C_\omega$.
By Lemma~\ref{LH2}, both $G[X]$ and $G[Y]$ are perfect. Therefore, $\chi(G[X])\leq \omega(G[X])\leq k$ and $\chi(G[Y])\leq \omega(G[Y])\leq k$. Moreover, $V(G)= V(A)\cup X\cup Y$. 

Now we may colour $a_i$ with $i$, and colour $X$ with colours from $\{\lfloor\frac{\omega}{2}\rfloor+1,\lfloor\frac{\omega}{2}\rfloor+2,\ldots,\lfloor\frac{\omega}{2}\rfloor+k\}$. If $k\leq \lfloor\frac{\omega}{2}\rfloor$, then we colour $Y$ with colours from $\{1,2,\ldots,\lfloor\frac{\omega}{2}\rfloor\}$. 
If  $k> \lfloor\frac{\omega}{2}\rfloor$, then we colour $Y$ with colours from $\{1,2,\ldots,\lfloor\frac{\omega}{2}\rfloor,\lfloor\frac{\omega}{2}\rfloor+k+1,\ldots,2k\}$. In both cases we obtain $\chi(G)\leq \max\{2k, \omega\}$.
\end{proof}

By Lemma~\ref{LH3}, if $k\leq 3$, then
\[
    \max \{2k,\omega\}\leq \max \{6,\omega\}\leq \left\lceil \frac{4}{3} \omega\right\rceil.
\]
Hence, in the remainder of this section, we may further assume that $k\geq 4$.

Let $B$ be an induced complete $k$-partite subgraph of $H$ of maximum order with parts $B_1, B_2, \ldots, B_k$. For the remainder of this section, unless specified otherwise, we adopt the following convention: for $1\le i\le\omega$ and $1\le j\le k$, $a_i$ (possibly with decorations such as $a_i'$ or $a_i^*$) denotes a vertex in $A_i$, and $b_j$ (possibly with decorations such as $b_j'$ or $b_j^*$) 
denotes a vertex in $B_j$. We first prove the following claim.

\begin{claim}\label{CH1}
If there exists a vertex $v\in V(A)$ such that $N_B(v)$ intersects at least two parts of $B$, then $\chi(G)\leq \omega +2$.
\end{claim}

\begin{proof}
By symmetry, we may let $a_1\in A_1$ such that $N_B(a_1)$ intersects at least two parts of $B$. 
Since $G$ is HVN-free and $k\geq 4$, similar to the proof of Lemma~\ref{LH1}, we can prove that $a_1$ is complete to $k-1$ parts of $B$. Without loss of generality, assume that $a_1$ is complete to $B_1,B_2,\ldots,B_{k-1}$. Now $\{a_1\}\cup B_1\cup B_2\cup \cdots\cup B_{k-1}$ induces a complete $k$-partite graph. For $i\in[k]$, let $b_i\in B_i$, and denote by $D$ the clique $\{b_1,b_2,\ldots,b_{k-1}\}$.

We first observe that for every vertex $x\in C_2\cup C_3\cup\cdots\cup C_\omega$, either $x$ is complete to $D$, or $x$ is anticomplete to $D$, or $|N_D(x)|=1$. Otherwise, suppose $\{b_i,b_j\}\subseteq N_D(x)$ and $b_m\notin N_D(x)$, where $i, j, m\in [k-1]$. Since $x\nsim a_1$, $\{a_1,b_i,b_j,b_m,x\}$ induces an HVN, a contradiction.

Moreover, we assert that if there exist vertices $x, y \in C_2\cup C_3\cup\cdots\cup C_\omega$ such that $x$ is complete to $D$ and either $y$ is complete to $D$ or $|N_D(y)|=1$, then $x\nsim y$. 
Otherwise, suppose that $x\sim y$. If both $x$ and $y$ are complete to $D$, then $\{x,y\}\cup D$ would be a clique of order $k+1$, contradicting $\omega(H)=k$. If $x$ is complete to $D$ and $|N_D(y)|=1$, 
then, without loss of generality, assume $N_D(y)=\{b_1\}$. In this case, $\{x,b_1,b_2,b_3,y\}$ would induce an HVN, a contradiction. Next, we show that 
\begin{equation}\label{E1}
    \text{there exists a union of $\omega-2$ sets among $C_2,C_3,\ldots,C_\omega$ that forms a stable set,}
\end{equation}
by considering the following cases.

\medskip

\noindent\textbf{Case 1.} There exists a vertex in $C_2\cup C_3\cup\cdots\cup C_\omega$ that is anticomplete to $D$.

By symmetry, assume that $x\in C_2$ is anticomplete to $D$, and let $a_2\in N_A(x)$. Let $y\in C_3$, and let $a_3\in N_A(y)$. We assert that $y$ is complete to $D$. Otherwise, $|N_D(y)|\le 1$, and since 
$|D|=k-1\ge 3$, there exist two vertices $z,z'\in D$ such that 
$y$ is anticomplete to $\{z,z'\}$. If $x\nsim y$, then $\{z,z^\prime,x,a_2,a_3,y\}$ induces a $P_2\cup P_4$; if $x\sim y$, then $\{z,z^\prime,a_4,a_2,x,y\}$ induces a $P_2\cup P_4$. In both cases, we obtain a contradiction. Therefore, $C_3$ is complete to $D$. 
Hence, if $x\in C_j$ ($2\le j\le k-1$) is anticomplete to $D$, then 
$(C_2\cup C_3\cup\cdots\cup C_\omega)\setminus C_j$ is complete to $D$. 
Thus, $(C_2\cup C_3\cup\cdots\cup C_\omega)\setminus C_j$ is stable, and hence 
\eqref{E1} holds.

\smallskip

\noindent\textbf{Case 2.} 
Every vertex in $C_2\cup C_3\cup\cdots\cup C_\omega$ is not anticomplete to $D$.

In this case, for every vertex $x\in C_2\cup C_3\cup\cdots\cup C_\omega$, 
either $x$ is complete to $D$ or $|N_D(x)|=~1$.

Moreover, by the previous argument, if $x$ is complete to $D$, 
then $x$ is anticomplete to all other vertices in $C_2\cup C_3\cup\cdots\cup C_\omega$. 
Therefore, we only need to consider the adjacency between vertices 
in $C_2\cup C_3\cup\cdots\cup C_\omega$ that have exactly one neighbour in $D$. 
Since $G$ is HVN-free, every such vertex $x$ is nonadjacent to $b_k$. 
We now consider the following two subcases according to the value of $\omega$.

\smallskip

\noindent\textbf{Case 2.1.} $\omega= 4$.

In this case, $a_1\nsim b_4$; otherwise, $\{a_1,b_1,b_2,b_3,b_4\}$ would induce a $K_5$, a contradiction.

We first show that each of $C_2$, $C_3$, and $C_4$ is stable. 
Let $x,y\in C_2$ with $|N_D(x)|=|N_D(y)|=1$. 
Since $|D|=k-1\ge3$, there exists a vertex $z\in D$ that is anticomplete to $\{x,y\}$. 
If $x\sim y$, then $\{x,y,a_4,a_1,z,b_4\}$ would induce a $P_2\cup P_4$, a contradiction. 
Thus $x\nsim y$, and hence $C_2$ is stable. 
By symmetry, each of $C_2$, $C_3$, and $C_4$ is stable.

If $b_4\in C_0\cup C_1$, then $C_2\cup C_3\cup C_4$ is stable. 
Otherwise, since each of $C_2$, $C_3$, and $C_4$ is stable, we may assume that there exist $x\in C_2$ and $y\in C_3$ such that $|N_D(x)|=|N_D(y)|=1$ and $x\sim y$. 
Let $a_2\in N_A(x)$ and $a_3\in N_A(y)$. 
Since $|D|=k-1\ge3$, there exists a vertex $z\in D$ that is anticomplete to $\{x,y\}$. 
Then $\{x,y,a_4,a_1,z,b_4\}$  induces a $P_2\cup P_4$, a contradiction. 
Hence $C_2$ is anticomplete to $C_3$. 
By symmetry, any two among $C_2$, $C_3$, and $C_4$ are anticomplete to each other, and hence $C_2\cup C_3\cup C_4$ is stable.

If $b_4\in C_2$, then by a similar argument we can show that both $C_2\cup C_3$ and 
$C_2\cup C_4$ are stable. Therefore, by symmetry, if $b_4\in C_i$ for some $2\le i\le 4$, 
then $C_i\cup C_j$ is stable for every $j\in\{2,3,4\}\setminus\{i\}$. 
Thus, there exists a union of two sets among $C_2$, $C_3$, and $C_4$ that forms a stable set, and hence \eqref{E1} holds.

\smallskip

\noindent\textbf{Case 2.2.} $\omega\geq 5$.

In this case, we will show that at most one of $C_2, C_3,\ldots, C_\omega$ contains vertices having exactly one neighbour in $D$, implying that \eqref{E1} holds. Suppose otherwise. By symmetry, let $x \in C_2$ with $N_D(x) = \{z\}$ and $a_2 \in N_A(x)$, and let $y \in C_3$ with $N_D(y) = \{z'\}$ and $a_3 \in N_A(y)$. 

If $z=z^\prime$, then there exist two vertices $u,v\in D$ anticomplete to $x$ and $y$ since $|D|=k-1\geq 3$. If $x\sim y$, then $\{u,v,a_4,a_2,x,y\}$ induces a $P_2\cup P_4$; if $x\nsim y$, then $\{u,v,x,a_2,a_3,y\}$ induces a $P_2\cup P_4$. Both cases lead to a contradiction. If $z\neq z^\prime$, then there exists a vertex $w\in D$ that is anticomplete to $x$ and $y$. If $x\sim y$, then $\{a_4,a_5,w,z,x,y\}$ induces a $P_2\cup P_4$; if $x\nsim y$, then $\{a_4,a_5,x,z,z^\prime,y\}$ induces a $P_2\cup P_4$. Both cases yield a contradiction.

\medskip

Hence,  there exists a union of $\omega - 2$ sets among $C_2, C_3, \ldots, C_\omega$ that forms a stable set. By symmetry, we suppose that $C_3\cup \cdots\cup C_\omega$ is stable. Then we partition $V(G)$ into three sets: $V(A)$, $C_0\cup C_1\cup C_2$, and $C_3\cup\cdots\cup C_\omega$. By Lemma~\ref{LH2}, $G[C_0\cup C_1\cup C_2]$ is perfect, and hence can be coloured with at most $\omega$ colours.

Now we may colour $A_i$ with colour $i$, and colour $C_0\cup C_1\cup C_2$ with at most $\omega$ colours from $[w+2]\setminus\{1,2\}$, and colour $C_3\cup\cdots\cup C_\omega$ with colour $1$. Hence, $\chi(G)\leq \omega +2$. This completes the proof of Claim~\ref{CH1}.
\end{proof}

Since for $\omega\geq 4$ we have $w+2\leq \lceil\frac{4}{3}\omega\rceil $.     By Claim~\ref{CH1}, we may assume that for every vertex $v\in V(A)$, $N_B(v)$ intersects at most one part of $B$. Thus, 
\begin{equation}\label{E2}
    \text{$[\{a_1,a_2,\ldots,a_\omega\},\{b_1,b_2,\ldots,b_k\}]$ is a matching in $G$.}
    \tag{$\ast$}
\end{equation}

Since $H$ is HVN-free, similarly to Lemma~\ref{LH1}, we have the following lemma.

\begin{lemma}\label{LH4}
If $k\geq 4$, then for every vertex $v\in V(H)\setminus V(B)$, $N_B(v)$ intersects at most one part of $B$.
\end{lemma}

By Lemma~\ref{LH1} and Lemma~\ref{LH4}, we partition $V(H)\setminus V(B)$ into the following subsets, where $i\in [\omega]$ and $j\in [k]$.
\begin{align*}
    R_j &=\{v\in V(H)\setminus V(B) \mid N_A(v)=\emptyset, N_B(v)\neq\emptyset, N_B(v)\subseteq B_j \},\\
    S_i^j &=\{v\in V(H)\setminus V(B) \mid N_A(v)\neq\emptyset, N_B(v)\neq\emptyset, N_A(v)\subseteq A_i, N_B(v)\subseteq B_j\},\\  
    T_i &=\{v\in V(H)\setminus V(B) \mid N_A(v)\neq \emptyset, N_B(v)=\emptyset, N_A(v)\subseteq A_i  \},\\
    Z &=\{v\in V(H)\setminus V(B) \mid N_A(v)=\emptyset,  N_B(v)=\emptyset \}.
\end{align*}
By Lemma~\ref{LH2}, each of the sets $R_j$, $S_i^j$, $T_i$, and $Z$ induces a perfect graph. We then define
\[
R = \bigcup_{j=1}^{k} R_j,\quad 
S = \bigcup_{j=1}^k\bigcup_{i=1}^{\omega} S_i^j,\quad 
T = \bigcup_{i=1}^{\omega} T_i.
\]
Note that 
\[
V(G)=V(A)\cup V(B)\cup R\cup S\cup T\cup Z.
\]
If we define
\begin{align*}
    C_j^\prime &=\{v\in V(H)\setminus V(B) \mid N_B(v)\neq \emptyset \text{ \,and\, }N_B(v)\subseteq B_j   \},\\
    C_0^\prime &=\{v\in V(H)\setminus V(B) \mid N_B(v)=\emptyset\},
\end{align*}
then Table 1 illustrates the relations among the sets in the partition introduced earlier.

\begin{table}[h]
\centering
\scalebox{0.9} 
{
\renewcommand{\arraystretch}{2} 
\setlength{\tabcolsep}{15pt} 
\begin{tabular}{c|cccccc}
\rule{0pt}{3ex} 
$\cap$ & $C_0'$ & $C_1'$ & $C_2'$ & $C_3'$ & $\cdots$ & $C_k'$ \\
\hline
\rule{0pt}{3ex}
$C_0$ & $Z$ & $R_1$ & $R_2$ & $R_3$ & $\cdots$ & $R_k$ \\
$C_1$ & $T_1$ & $S_1^1$ & $S_1^2$ & $S_1^3$ & $\cdots$ & $S_1^k$ \\
$C_2$ & $T_2$ & $S_2^1$ & $\ddots$ & & & $\vdots$ \\
$C_3$ & $T_3$ & $S_3^1$ & & $\ddots$ & & $\vdots$ \\
$\vdots$ & $\vdots$ & $\vdots$ & & & $\ddots$ & $\vdots$ \\
$C_\omega$ & $T_\omega$ & $S_\omega^1$ & $\cdots$ & $\cdots$ & $\cdots$ & $S_\omega^k$ \\
\end{tabular}
}
\caption{Partitions of $V(H)\backslash V(B)$}
\label{tab: Partition}
\end{table}

\begin{claim}\label{C2}
Let $x,y$ be two vertices in $R\cup S\cup T\cup Z$.
If $N_{A\cup B}(x)\neq N_{A\cup B}(y)$, and the set $N_{A\cup B}(x)\cup N_{A\cup B}(y)$ intersects at most three parts among $A_1,\ldots,A_\omega,B_1,\ldots,B_k$, then $x\nsim y$.
\end{claim}
\begin{proof}
Suppose, to the contrary, that there exist vertices $x,y\in R\cup S\cup T\cup Z$ satisfying the above conditions such that $x\sim y$. By Lemmas~\ref{LH1} and~\ref{LH4}, we may distinguish between the following two cases.

\medskip

\noindent\textbf{Case 1.} $N_{A\cup B}(x)\cup N_{A\cup B}(y)$ intersects two parts among $A_1,\ldots,A_\omega$ or two parts among $B_1,\ldots,B_k$.

Suppose $N_{A\cup B}(x)\cup N_{A\cup B}(y)$ intersects two parts among $A_1,\ldots,A_\omega$ (the case where it intersects two parts among $B_1,\ldots,B_k$ is analogous). Since $N_{A\cup B}(x)\cup N_{A\cup B}(y)$ intersects at most three parts among $A_1,\ldots,A_\omega,B_1,\ldots,B_k$, it intersects at most one part among $B_1,\ldots,B_k$. By symmetry, we may assume that $x\in C_1\cap (C_0^\prime \cup C_1^\prime)$ and $y\in C_2\cap (C_0^\prime \cup C_1^\prime)$. Hence, $\{x,y\}$ is anticomplete to $A_3\cup \cdots\cup A_\omega\cup B_2\cup \cdots\cup B_k$. By \eqref{E2}, $[\{a_3,\ldots,a_\omega\},\{b_2,\ldots,b_k\}]$ is a matching. Since $\omega,k\geq 4$ and $G$ is $(P_2\cup P_4)$-free, Lemma~\ref{L0} implies that $\{a_3,\ldots,a_\omega\}$ is anticomplete to $\{b_2,\ldots,b_k\}$. Let $a_1\in N_A(x)$. Since $k\geq 4$, there exist vertices $b_i,b_j\in \{b_2,\ldots,b_k\}$ that are anticomplete to $a_1$. Consequently, $\{b_i,b_j,a_3,a_1,x,y\}$ induces a $P_2\cup P_4$, a contradiction.

\smallskip

\noindent\textbf{Case 2.} $N_{A\cup B}(x)\cup N_{A\cup B}(y)$ intersects at most one part among $A_1,\ldots,A_\omega$ and at most one part among $B_1,\ldots,B_k$.

By symmetry, we may assume that $x,y\in (C_0\cup C_1)\cap (C_0^\prime \cup C_1^\prime)$. Thus, $\{x,y\}$ is anticomplete to $A_2\cup \cdots\cup A_\omega\cup B_2\cup \cdots\cup B_k$. By \eqref{E2}, $[\{a_2,\ldots,a_\omega\},\{b_2,\ldots,b_k\}]$ is a matching. Since $\omega,k\geq 4$ and $G$ is $(P_2\cup P_4)$-free, Lemma~\ref{L0} implies that $\{a_2,\ldots,a_\omega\}$ is anticomplete to $\{b_2,\ldots,b_k\}$. Since $N_{A\cup B}(x)\neq N_{A\cup B}(y)$, it follows that  $N_{A}(x)\neq N_{A}(y)$ or $N_{B}(x)\neq N_{B}(y)$. Suppose $N_{A}(x)\neq N_{A}(y)$ (the case where $N_{B}(x)\neq N_{B}(y)$ is analogous). Then there exists a vertex $a_1\in A_1$ such that either $a_1\sim x$ and $a_1\nsim y$, or $a_1\nsim x$ and $a_1\sim y$. By symmetry, assume that $a_1\sim x$ and $a_1\nsim y$. Since $k\geq 4$, there exist vertices $b_i,b_j\in \{b_2,\ldots,b_k\}$ that are anticomplete to $a_1$. Consequently, $\{b_i,b_j,a_2,a_1,x,y\}$ induces a $P_2\cup P_4$, a contradiction.

\medskip

This completes the proof of Claim~\ref{C2}
\end{proof}

We now conclude Section $3$ by presenting several useful properties of the subsets in the partition defined earlier. In particular, Properties (P1)–(P4) follow directly from Claim \ref{C2}.

\smallskip

\noindent(P1) $Z$ is anticomplete to $R\cup S\cup T$.

\medskip

\noindent(P2) $R$ is anticomplete to $S\cup T\cup Z$, and $R_i$ is anticomplete to $R_j$ for distinct $i,j \in [k]$.

\medskip

\noindent(P3) $T$ is anticomplete to $R\cup S\cup Z$, and $T_i$ is anticomplete to $T_j$ for distinct $i,j \in [\omega]$.

\medskip

\noindent(P4) $S$ is anticomplete to $R\cup T\cup Z$. 
Moreover, $S_i^p$ is anticomplete to $S_i^q$ for every $i\in[\omega]$ and distinct $p,q\in[k]$; and $S_i^p$ is anticomplete to $S_j^p$ for every $p\in[k]$ and distinct $i,j\in[\omega]$.

\medskip

\noindent(P5) If $S_i^p$ is not stable, then $(A_1\cup\cdots\cup A_\omega)\setminus A_i$ is anticomplete to $(B_1\cup\cdots\cup B_k)\setminus B_p$, and $S_i^p$ is anticomplete to $S_j^q$ for distinct $p,q \in [k]$ and distinct $i,j \in [\omega]$.
\begin{proof}
By symmetry, it suffices to show that (P5) holds for $i=p=1$ and $j=q=2$. Then we may let $x,y\in S_1^1$ and $x\sim y$. By \eqref{E2}, $[\{a_2,\ldots,a_\omega\},\{b_2,\ldots,b_k\}]$ is a matching. Since $\omega,k\geq 4$ and $G$ is $(P_2\cup P_4)$-free, Lemma~\ref{L0} implies that $\{a_2,\ldots,a_\omega\}$ is anticomplete to $\{b_2,\ldots,b_k\}$. Hence, $A_2\cup \cdots\cup A_\omega$ is anticomplete to $B_2\cup \cdots\cup B_k$. 

Next, we show that $S_1^1$ is anticomplete to $S_2^2$. To the contrary, suppose that there exist $u\in S_1^1$ and $v\in S_2^2$ such that $u\sim v$. Let $a_2\in N_A(v)$. Then $\{b_3,b_4,a_3,a_2,v,u\}$ induces a $P_2\cup P_4$, a contradiction.
\end{proof}

\noindent(P6) Suppose $\omega = 4$. If $S_i^p$ is not anticomplete to $S_j^q$, then $S_i^p$ is anticomplete to $S_j^{q'}$ and $S_{j'}^q$ for pairwise distinct $p, q, q' \in [4]$ and pairwise distinct $i, j, j' \in [4]$.
\begin{proof}
By symmetry, it suffices to show that if $S_1^1$ is not anticomplete to $S_2^2$, then $S_1^1$ is anticomplete to $S_2^3$ and $S_3^2$. 
Let $x \in S_1^1$ and $y \in S_2^2$ with $x \sim y$, and let $a_1 \in N_A(x)$, $b_1 \in N_B(x)$, $a_2 \in N_A(y)$, and $b_2 \in N_B(y)$. 

If $b_3$ is anticomplete to $\{a_3,a_4\}$, then $|[\{b_3,b_4\},\{a_1,a_2\}]|=2$ and $|[\{a_3,a_4\},\{b_1,b_2\}]|=~2$. Otherwise, if $|[\{b_3,b_4\},\{a_1,a_2\}]|\leq 1$, then there exist $a_i\in \{a_1,a_2\}$ and $a_j\in \{a_3,a_4\}$ such that both $a_i$ and $a_j$ are 
anticomplete to $\{b_3,b_4\}$, and hence $\{b_3,b_4,a_j,a_i,x,y\}$ induces a $P_2\cup P_4$, a contradiction. Similarly, if $|[\{a_3,a_4\},\{b_1,b_2\}]|\leq 1$, then there exists $b_p\in \{b_1,b_2\}$ anticomplete to $\{a_3,a_4\}$, and hence $\{a_3,a_4,b_3,b_p,x,y\}$ induces a $P_2\cup P_4$, a contradiction.

If $b_3$ is not anticomplete to $\{a_3,a_4\}$, then since $\{x,y,b_3,b_4,a_3,a_4\}$ cannot induce a $P_2\cup P_4$, it follows that $|[\{b_3,b_4\},\{a_3,a_4\}]|=2$. 

We now show that $S_1^1$ is anticomplete to $S_2^3$; the argument for $S_1^1$ being anticomplete to $S_3^2$ is analogous.  Suppose, to the contrary, that there exists a vertex $z\in S_2^3$ such that $x\sim z$. If $|[\{b_3,b_4\},\{a_1,a_2\}]|=2$ and $|[\{a_3,a_4\},\{b_1,b_2\}]|=2$, then $\{a_3,a_4\}$ is anticomplete to $b_4$ and $\{a_1,a_2\}$ is anticomplete to $b_2$. Moreover, there exist vertices $u\in \{a_3,a_4\}$ nonadjacent to $b_2$ and  $v\in \{a_1,a_2\}$ nonadjacent to $b_4$. Thus, $\{b_2,b_4,u,v,x,z\}$ induces a $P_2\cup P_4$, a contradiction. 
If $|[\{b_3,b_4\},\{a_3,a_4\}]|=2$, then $\{a_1,a_2\}$ is anticomplete to $b_4$ and $\{a_3,a_4\}$ is anticomplete to $b_2$. Moreover, there exist vertices $u\in \{a_3,a_4\}$ nonadjacent to $b_4$ and $v\in \{a_1,a_2\}$ nonadjacent to $b_2$. Therefore, $\{b_2,b_4,u,v,x,z\}$ induces a $P_2\cup P_4$, a contradiction.
\end{proof}

\noindent(P7) Suppose $\omega = 4$. Let $S'$ be the union of all stable sets $S_i^j$ in $S$. Then the vertex set of each component in $G[S^\prime]$ can be partitioned into at most four stable sets, each contained in a distinct $S_i^j$, with all $i$’s and all $j$’s being distinct within that component.
\begin{proof}
This follows directly from (P4) and (P6).
\end{proof}

\noindent(P8) If $\omega\geq5$, then $S_i^p$ is anticomplete to $S_j^q$ for distinct $p,q \in [k]$ and distinct $i,j \in [\omega]$.
\begin{proof}
By symmetry, it suffices to show that $S_1^1$ is anticomplete to $S_2^2$. Suppose to the contrary that there exists vertices $x\in S_1^1$ and $y\in S_2^2$ such that $x\sim y$. Let $b_1\in N_B(x)$. By \eqref{E2}, $[\{a_3,\ldots,a_\omega\},\{b_3,\ldots,b_k\}]$ is a matching.  Since $\omega\ge 5$, $k\ge 4$ and $G$ is $(P_2\cup P_4)$-free, Lemma~\ref{L0} implies that $\{a_3,\ldots,a_\omega\}$ is anticomplete to $\{b_3,\ldots,b_k\}$. Since $\omega\geq 5$, there exist vertices $a_i,a_j\in \{a_3,\ldots,a_\omega\}$ that are anticomplete to $b_1$. Consequently, $\{a_i,a_j,b_3,b_1,x,y\}$ induces a $P_2\cup P_4$, a contradiction.
\end{proof}

\section{The optimal $\chi$-bound}

In this section, we use the notation introduced in Section~3.

\begin{proof}[\normalfont\textbf{Proof of Theorem 1.1}]

Let $G$ be a ($P_2 \cup P_4$, HVN)-free graph with $\omega(G) \ge 2$. 
By Lemma~\ref{LP2P4}, it suffices to prove the theorem for $\omega(G) \ge 4$. 
Moreover, by Lemma~\ref{LH3}, we may further assume that $k \ge 4$.

Recall that, by Lemma~\ref{LH2}, each of the sets $R_j$, $S_i^j$, $T_i$, and $Z$ induces a perfect graph. 
Hence, each of them can be properly coloured using at most $\omega$ colours. Moreover, by Claim~\ref{CH1}, we may assume that for every vertex $v\in V(A)$, $N_B(v)$ intersects at most one part of $B$. 
We first show that $G$ is $\lceil \frac{4}{3}\omega \rceil$-colourable by considering two cases according to the value of $\omega$.

\medskip
\noindent\textbf{Case 1.} $\omega =4$.
\medskip

In this case, $k = 4$. We show that $G$ admits a $6$-colouring. 
\smallskip

We first colour $V(A)$.
\begin{itemize}
\item Colour $A_i$ with colour $i$ for $i \in [4]$.
\end{itemize}

Then we colour $V(B)$ and $S$, distinguishing according to whether some $S_i^j$ is stable.

\medskip
If there exists a set $S_i^j$ that is not stable, we may assume without loss of generality that $S_1^1$ is not stable. 
By (P5), $A_2 \cup A_3 \cup A_4$ is anticomplete to $B_2 \cup B_3 \cup B_4$. 
We colour $V(B)$ and $S$ as follows.
\begin{itemize}
\item Colour $B_1$ with colour $5$, and colour $B_2$, $B_3$, and $B_4$ with colours $2$, $3$, and $4$, respectively.
\item By (P5), colour each $S_i^j$ that is not stable with at most 4 colours from $[6] \setminus (c(A_i) \cup c(B_j))$.
\item By (P5) and (P7), colour each component of $G[S^\prime]$ (recall that $S^\prime$ is the union of all stable sets $S_i^j$ in $S$) with at most four suitable colours from $[6]$. 
To justify this, we construct a bipartite graph that demonstrates the existence of such a colouring. 
Let $X$ denote the set of stable sets in a component of $G[S^\prime]$ as described in (P7). 
For each $X_t \in X$, let $Y_t \subseteq [6]$ be the set of colours that can be assigned to $X_t$; 
for example, if $X_t \subseteq S_i^j$, then $Y_t = [6] \setminus (c(A_i) \cup c(B_j))$. 
Let 
\[
Y = \bigcup_{X_t \in X} Y_t.
\]
We now regard $X$ and $Y$ as the two vertex sets of a bipartite graph $Q$ with bipartition $(X, Y)$. 
Each vertex $X_t\in X$ represents a stable set in the component, and each vertex $y\in Y$ represents a colour in $[6]$. 
Vertices $X_t$ and $y$ are adjacent in $Q$ if and only if $y\in Y_t$; equivalently, $N_Q(X_t)=Y_t$ for all $X_t\in X$. 
By Lemma~\ref{HALL} applied to $Q$, there exists a matching saturating $X$.
\end{itemize}

If every $S_i^j$ is stable, we colour $V(B)$ and $S$ as follows.
\begin{itemize}
\item Colour $B_1$ with colours from $\{1,2\}$ and $B_2$ with colours from $\{3,4\}$. 
This is possible since, by Lemma~\ref{LH1}, the neighbours of each vertex in $B_i$ intersect at most one part of $A$. Hence we may assign one colour to the vertices in $B_i$ having no neighbour in the part of $A$ that receives the same colour, and the other colour to those having such a neighbour. Finally, colour $B_3$ and $B_4$ with colours $5$ and $6$, respectively.
\item By (P7), colour each component of $G[S]$ with at most four suitable colours from $[6]$. 
The same argument used for $G[S^\prime]$ in the case where some $S_i^j$ is not stable applies 
here as well.
\end{itemize}

Finally, we colour $Z$, $R$, and $T$.
\begin{itemize}
\item By (P1), colour $Z$ with at most 4 colours from $[6]$.
\item By (P2), colour $R_i$ with at most 4 colours from $[6]\setminus c(B_i)$ for $i \in [4]$.
\item By (P3), colour $T_i$ with at most 4 colours from $[6] \setminus c(A_i)$ for $i \in [4]$.
\end{itemize}

\noindent\textbf{Case 2.} $\omega \geq 5$.
\medskip

\noindent\textbf{Case 2.1. } There exists a set $S_i^j$ that is not stable.
\medskip

Without loss of generality, we may assume that $S_1^1$ is not stable. By (P5), $A_2 \cup \cdots \cup A_\omega$ is anticomplete to $B_2 \cup \cdots \cup B_k$. We colour $G$ as follows.

\begin{itemize}
\item Colour $A_i$ with colour $i$ for $i\in [\omega]$.
\item Colour $B_1$ with colour $\omega +1$ and $B_i$ with colour $i$ for $i\in [k]\setminus \{1\}$.
\item By (P4) and (P8), colour each $S_i^j$ with at most $\omega$ colours from $\lceil\frac{4}{3}\omega\rceil\setminus (c(A_i)\cup c(B_j))$ for $i\in [\omega]$ and $j\in [k]$.
\item By (P1), colour $Z$ with at most $\omega$ colours from $\lceil\frac{4}{3}\omega\rceil$.
\item By (P2), colour $R_i$ with at most $\omega$ colours from $\lceil\frac{4}{3}\omega\rceil\setminus c(B_i)$ for $i\in [k]$.
\item By (P3), colour $T_i$ with at most $\omega$ colours from $\lceil\frac{4}{3}\omega\rceil\setminus c(A_i)$ for $i\in [\omega]$.
\end{itemize}

\noindent\textbf{Case 2.2. } Every $S_i^j$ is stable.
\medskip

We first colour $V(A)$ as follows.
\begin{itemize}
    \item For $1 \leq i \leq \omega - \lceil \frac{1}{3}\omega \rceil$, colour $A_i$ with colour $i$.
    \item For $\omega - \lceil \frac{1}{3}\omega \rceil + 1 \leq i \leq \omega$, colour $A_i$ according to the following rules:
    \begin{itemize}
        \item If $B_i \neq \emptyset$ (equivalently, $i \leq k$), then colour the vertices in $A_i$ that has a neighbour in $B_i$ with colour $i + \lceil \frac{1}{3}\omega \rceil$, and colour those that are anticomplete to $B_i$ with colour $i$.
        \item If $B_i = \emptyset$ (equivalently, $i > k$), then colour $A_i$ with colour $i$.
    \end{itemize}
\end{itemize}

Then we colour $V(B)$ as follows.
\begin{itemize}
\item For $1 \leq i \leq \lceil \frac{1}{3}\omega \rceil$, colour $B_i$ with colour $\omega + i$, since $B_i$ is anticomplete to those vertices in $V(A)$ that have received colour $\omega + i$.
\item For $\omega - \lceil \frac{1}{3}\omega \rceil + 1 \leq i \leq \omega$, colour $B_i$ with colour $i$, since $B_i$ is anticomplete to those vertices in $V(A)$ that have received colour $i$.
\item For $\lceil \frac{1}{3}\omega \rceil + 1 \leq i \leq \omega - \lceil \frac{1}{3}\omega \rceil$, colour each $B_i$ with at most two colours from $\{1,2,\ldots,\omega - \lceil \frac{1}{3}\omega \rceil\}$, ensuring that distinct $B_i$'s receive pairwise disjoint colour sets.
This is possible because the colours in $\{1,2,\ldots,\omega - \lceil \frac{1}{3}\omega \rceil\}$ have not been used on the other parts of $V(B)$, 
and the number of colours that may be required in this step does not exceed the number of colours in $\{1,2,\ldots,\omega - \lceil \frac{1}{3}\omega \rceil\}$.

Indeed, the maximum number of colours that may be needed here is
\[
    2(\omega - \lceil \frac{1}{3}\omega \rceil - (\lceil \frac{1}{3}\omega \rceil + 1) + 1),
\]
while the number of available colours is
\[
    \omega - \lceil \frac{1}{3}\omega \rceil.
\]
Since
\[
    2(\omega - \lceil \frac{1}{3}\omega \rceil - (\lceil \frac{1}{3}\omega \rceil + 1) + 1)
    \leq \omega - \lceil \frac{1}{3}\omega \rceil,
\]
the colouring is therefore valid.
\end{itemize}

Finally, we colour $S$, $Z$, $R$, and $T$ as follows.
\begin{itemize}
\item By (P4) and (P8), colour each $S_i^j$ with one colour from $\lceil\frac{4}{3}\omega\rceil\setminus (c(A_i)\cup c(B_j))$ for $i\in [\omega]$ and $j\in [k]$.
\item By (P1), colour $Z$ with at most $\omega$ colours from $\lceil\frac{4}{3}\omega\rceil$.
\item By (P2), colour $R_i$ with at most $\omega$ colours from $\lceil\frac{4}{3}\omega\rceil\setminus c(B_i)$ for $i\in [k]$.
\item By (P3), colour $T_i$ with at most $\omega$ colours from $\lceil\frac{4}{3}\omega\rceil\setminus c(A_i)$ for $i\in [\omega]$.
\end{itemize}

Next, we prove that the bound is optimal when $\omega \geq 4$. 
\medskip

We consider a graph $G$ of order $2\omega^2$, with $V(G) = V(A) \cup V(B)$, 
where $A$ and $B$ are defined as before. 
Specifically, both $A$ and $B$ are induced complete $\omega$-partite subgraphs of $G$ 
satisfying $|V(A)| = |V(B)| = \omega^2$, 
and each part of $A$ and of $B$ contains exactly $\omega$ vertices. 
For $i \in [\omega]$, define
\[
A_i = \{a_i^1, a_i^2, \ldots, a_i^\omega\}
\quad \text{and} \quad
B_i = \{b_i^1, b_i^2, \ldots, b_i^\omega\}.
\]
We define the edges between $V(A)$ and $V(B)$ as follows.
for each $i, j \in [\omega]$,
\[
N_B(a_j^i) = \{b_i^j\}
\quad \text{and} \quad
N_A(b_i^j) = \{a_j^i\}.
\]
Thus, for each $i \in [\omega]$, every vertex of $A_i$ has exactly one neighbour in $V(B)$, 
and the neighbours of distinct vertices in $A_i$ are in distinct parts $B_1, B_2, \ldots, B_\omega$. 
Similarly, for each $i \in [\omega]$, every vertex of $B_i$ has exactly one neighbour in $V(A)$, 
and the neighbours of distinct vertices in $B_i$ are in distinct parts $A_1, A_2, \ldots, A_\omega$. 
Note that $[V(A), V(B)]$ is a perfect matching of $G$. Moreover, one can verify that $G$ is ($P_2\cup P_4$, HVN)-free.

Suppose, to the contrary, that $G$ is $(\lceil \frac{4}{3}\omega \rceil - 1)$-colourable. 
Since $A$ is a complete $\omega$-partite graph, at least $\omega$ colours from $[\lceil \frac{4}{3}\omega \rceil - 1]$ must be used on $V(A)$, and distinct parts of $A$ must receive pairwise disjoint sets of colours. 
Hence, there exist at least 
\[
\omega - \big(\lceil \frac{4}{3}\omega \rceil - 1 - \omega\big) = \omega - \lceil \frac{1}{3}\omega \rceil + 1
\]
parts of $A$ that are assigned only one colour. Without loss of generality, assume that all vertices of $A_i$ are assigned colour $i$ for $i \in [\omega - \lceil \frac{1}{3}\omega \rceil + 1]$. 
Hence, for each $j \in [\omega]$,
\[
[\omega - \lceil \frac{1}{3}\omega \rceil + 1] \subseteq c(N_A(B_j)).
\]
Similarly, there exist at least $\omega - \lceil \frac{1}{3}\omega \rceil + 1$ parts of $B$ 
that are assigned only one colour. 
Thus, these parts of $B$ cannot be assigned any colours from 
$[\omega - \lceil \frac{1}{3}\omega \rceil + 1]$. 
Hence, the number of remaining colours that are available for these parts of $B$ is
\[
(\lceil \frac{4}{3}\omega \rceil - 1) - (\omega - \lceil \frac{1}{3}\omega \rceil + 1)
= 2\lceil \frac{1}{3}\omega \rceil - 2.
\]
Since
\[
2\lceil \frac{1}{3}\omega \rceil-2<\omega -\lceil \frac{1}{3}\omega \rceil+1,
\]
we obtain a contradiction.

This completes the proof of Theorem~\ref{THM2}
\end{proof}

\section*{Declaration of competing interest}

The authors declare that they have no known competing financial interests or personal relationships that could have appeared to influence the work reported in this paper.

\section*{Data availability}

No data was used for the research described in the article.

\end{document}